\documentclass[12pt]{article}
\usepackage{graphicx}
\usepackage{mathrsfs}
\usepackage{amsfonts,amsmath,amsthm, amssymb}
\usepackage{latexsym, euscript, epic, eepic, color}
\newtheorem{lemma}{Lemma}[section]
\newtheorem{theorem}{Theorem}[section]
\newtheorem{definition}{Definition}[section]
\newtheorem{proposition}{Proposition}[section]
\newtheorem{corollary}{Corollary}

\newtheorem{remark}{Remark}

\begin{document}

\title{Limit theorems related to beta-expansion and continued fraction expansion
\footnotetext {* Corresponding author}
}

\author{  Lulu Fang$^\dag$, Min Wu$^\dag$ and Bing Li$^{\dag, *}$\\
\small \it $\dag$ Department of Mathematics\\
\small \it South China University of Technology\\
\small \it Guangzhou 510640, P.R. China\\
\small \it E-mails: fanglulu1230@163.com, wumin@scut.edu.cn and scbingli@scut.edu.cn}
\date{}
\maketitle

\begin{center}
\begin{minipage}{120mm}
{\small {\bf Abstract.}
  Let $\beta > 1$ be a real number and $x \in [0,1)$ be an irrational number. Denote by $k_n(x)$ the exact number of partial quotients in the continued fraction expansion of $x$ given by the first $n$ digits in the $\beta$-expansion of $x$ ($n \in \mathbb{N}$). In this paper, we show a central limit theorem and a law of the iterated logarithm for the random variables sequence $\{k_n, n \geq 1\}$, which generalize the results of Faivre \cite{lesFai98} and Wu \cite{lesWu08} respectively from $\beta =10$ to any $\beta >1$.
}
\end{minipage}
\end{center}

\vskip0.5cm

\section{Introduction}
Let $\beta > 1$ be a real number and $T_\beta: [0,1) \longrightarrow [0,1)$ be the \emph{$\beta$-transformation} defined as
\begin{equation*}\label{T-beta}
T_\beta(x)= \beta x - \lfloor\beta x\rfloor,
\end{equation*}
where $\lfloor x\rfloor$ denotes the greatest integer not exceeding $x$. Then every $x \in [0,1)$ can be uniquely expanded into a finite or infinite series, i.e.,
\begin{equation}\label{beta expansion}
x = \frac{\varepsilon_1(x)}{\beta} + \frac{\varepsilon_2(x)}{\beta^2} + \cdots + \frac{\varepsilon_n(x)}{\beta^n} + \cdots,
\end{equation}
where $\varepsilon_1(x) = \lfloor\beta x\rfloor$ and $\varepsilon_{n+1}(x) = \varepsilon_1(T_\beta^nx)$ for all $n \geq 1$. We call the representation (\ref{beta expansion}) the \emph{$\beta$-expansion} of $x$ denoted by $(\varepsilon_1(x),\varepsilon_2(x),\cdots, \varepsilon_n(x),\cdots)$ and $\varepsilon_n(x), n \geq 1$ the \emph{digits} of $x$. Such an expansion was first introduced by R\'{e}nyi \cite{lesRen57}, who proved that there exists a unique $T_\beta$-invariant measure equivalent to the Lebesgue measure $\mathrm{P}$ when $\beta$ is not an integer; while it is known that the Lebesgue measure is $T_\beta$-invariant when $\beta$ is an integer. Furthermore, Gel'fond \cite{lesGel59} and Parry \cite{lesPar60} independently found the density formula for this invariant measure with respect to (w.r.t.) the Lebesgue measure. The arithmetic and metric properties of $\beta$-expansion were studied extensively in the literature, such as \cite{lesBla89, lesD.K02, lesF.W12, lesF.S92, lesHT08, lesL.P.W.W14, lesL.W08, lesSch97, lesSch80} and the references therein.

Now we turn our attention to continued fraction expansions. Let $T: [0,1) \longrightarrow [0,1)$ be the \emph{Gauss transformation} given by
\begin{equation*}\label{Gauss}
Tx =
\begin{cases}
\dfrac{1}{x} - \left\lfloor\dfrac{1}{x}\right\rfloor, &\text{if $x \in (0,1)$}; \\
0 , &\text{if $x=0$}.
\end{cases}
\end{equation*}
Then any real number $x \in [0,1)$ can be written as
\begin{equation}\label{continued fraction expansion}
x = \dfrac{1}{a_1(x) +\dfrac{1}{a_2(x) + \ddots +\dfrac{1}{a_n(x)+ \ddots}}},
\end{equation}
where $a_1(x) = \lfloor1/x\rfloor$ and $a_{n +1}(x) = a_1(T^nx)$ for all $n \geq 1$. The form (\ref{continued fraction expansion}) is said to be the \emph{continued fraction expansion} of $x$ and $a_n(x), n \geq 1$  are called the \emph{partial quotients} of $x$.  Sometimes we write the form (\ref{continued fraction expansion}) as $[a_1(x), a_2(x), \cdots, a_n(x), \cdots]$. For any $n \geq 1$, we denote by $\frac{p_n(x)}{q_n(x)}:= [a_1(x), a_2(x), \cdots, a_n(x)]$ the $n$-th \emph{convergent} of $x$, where $p_n(x)$ and $q_n(x)$ are relatively prime.  Clearly these convergents are rational numbers and $p_n(x)/q_n(x) \rightarrow x$ as $n \rightarrow \infty$ for all $x \in [0,1)$. More precisely,
\begin{equation}\label{diophantine}
\frac{1}{2q_{n+1}^2(x)} \leq \frac{1}{2q_n(x)q_{n+1}(x)} \leq \left|x-\frac{p_n(x)}{q_n(x)}\right| \leq \frac{1}{q_n(x)q_{n+1}(x)} \leq \frac{1}{q_n^2(x)}.
\end{equation}
This is to say that the speed of $p_n(x)/q_n(x)$ approximating to $x$ is dominated by $q_n^{-2}(x)$. So the denominator of the $n$-th convergent $q_n(x)$ plays an important role in the problem of Diophantine approximation. For more details about continued fractions, we refer the reader to a monograph of Khintchine \cite{lesKhi64}.

Ibragimov\cite{lesIbr61} proved that a central limit theorem holds for the sequence $\{q_n, n\geq 1\}$. Furthermore, Morita \cite{lesMor94} showed that the remainder in the central limit theorem is as we would expect $\mathcal{O}(n^{-1/2})$ (see also Misevi\v{c}ius \cite{lesMis70} and Philipp \cite{lesPhi70}), where $a_n = \mathcal{O}(b_n)$ denotes that there exists a constant $C>0$ such that $|a_n| \leq C \cdot b_n$ for all $n \geq 1$.

\begin{theorem}[\cite{lesIbr61}]\label{central limit theorem for qn}
For every $y \in \mathbb{R}$,
\[
  \lim_{n \to \infty}  \mathrm{P} \left\{x \in [0,1): \frac{\log q_n(x)- \frac{\pi^2}{12\log 2}n}{\sigma_1\sqrt{n}} \leq y \right\} = \frac{1}{\sqrt{2\pi}} \int_{-\infty}^y e^{-\frac{t^2}{2}}dt,
\]
where $\sigma_1 > 0$ is an absolute constant.
\end{theorem}

\begin{remark}
The constant $\sigma_1$ can be obtained by the generalized tranfer operators $L_s$, $s>1$ of Mayer \cite{lesMay90}, where $L_s$, $s>1$ are defined by
\[
L_s f(z) = \sum_{n = 1}^\infty \left(\frac{1}{n + z}\right)^s f\left(\frac{1}{n + z}\right)
\]
in a suitable space of holomorphic functions.  A particular property of these operators is to have a simple dominant eigenvalue $\lambda(s) >0$. Flajolet and Vall\'{e}e \cite{lesF.V98} pointed out that there is a beautiful expression of $\sigma_1$ using the dominant eigenvalue of $L_2$,
\begin{equation*}
\sigma_1^2 = \lambda^{\prime\prime}(2) - (\lambda^{\prime}(2))^2,
\end{equation*}
where $\lambda^{\prime}(s)$ and $\lambda^{\prime\prime}(s)$ denote the derivative and second derivative of $\lambda(s)$ respectively.
\end{remark}

Later, Philipp and Stackelberg \cite{lesP.S69} provided the following law of the iterated logarithm for the sequence $\{q_n, n \geq 1\}$.

\begin{theorem}[\cite{lesP.S69}]\label{q law}
For $\mathrm{P}$-almost all $x \in [0,1)$,
  \[
  \limsup\limits_{n \to \infty} \frac{\log q_n(x)- \frac{\pi^2}{12\log 2}n}{\sigma_1\sqrt{2n\log\log n}} = 1
  \]
  and
  \[
  \liminf\limits_{n \to \infty} \frac{\log q_n(x)- \frac{\pi^2}{12\log 2}n}{\sigma_1\sqrt{2n\log\log n}} = -1,
  \]
  where $\sigma_1 > 0$ is the same constant as in Theorem \ref{central limit theorem for qn}.
\end{theorem}

A natural question is if there exists some relationship between different expansions of some real number $x \in [0,1)$, for instance, its $\beta$-expansion and continued fraction expansion.
For any irrational number $x \in [0,1)$ and $n \geq 1$, we denote by $k_n(x)$ the exact number of partial quotients in the continued fraction expansion of $x$ given by the first $n$ digits in the $\beta$-expansion of $x$. That is,
\[
k_n(x) = \sup\left\{m \geq 0: J(\varepsilon_1(x),\cdots, \varepsilon_n(x)) \subset I(a_1(x), \cdots, a_m(x))\right\},
\]
where $J(\varepsilon_1(x),\cdots, \varepsilon_n(x))$ and $I(a_1(x), \cdots, a_m(x))$ are called the \emph{cylinders} of $\beta$-expansion and continued fraction expansion respectively (see Section 2). It is easy to check that
\begin{equation}\label{increasing sequence}
0 \leq k_1(x) \leq k_2(x) \leq \cdots\ \ \text{and}\ \ \lim\limits_{n \to \infty}k_n(x) = \infty.
\end{equation}
The quantity $k_n(x)$ was first introduced by Lochs \cite{lesLoc64} for $\beta = 10$ and has been extensively investigated by many mathematicians, see \cite{lesB.I08, lesD.F01, lesFai97, lesFai98, lesL.W08, lesWu06, lesWu08}.
Applying the result of Dajani and Fieldsteel \cite{lesD.F01} (Theorem 5) to $\beta$-expansion and continued fraction expansion, Li and Wu \cite{lesL.W08} obtained a metric result of $\{k_n, n \geq 1\}$, that is, for $\mathrm{P}$-almost all $x \in [0,1)$,
\begin{equation}\label{theorem eq}
\lim\limits_{n \to \infty}\frac{k_n(x)}{n} = \frac{6\log2\log\beta}{\pi^2}.
\end{equation}
The formula (\ref{theorem eq}) has been stated for $\beta=10$ by a pioneering result of Lochs \cite{lesLoc64}. Barreira and Iommi \cite{lesB.I08} proved that the irregular set of points $x \in [0,1)$ for which the limit in (\ref{theorem eq}) does not exist has Hausdorff dimension 1. Li and Wu \cite{lesL.W08} gave some asymptotic results of $k_n(x)/n$ for any irrational $x \in [0,1)$ not just a kind of almost all result (see also Wu \cite{lesWu06}). For the special case $\beta = 10$, some limit theorems of $\{k_n, n \geq 1\}$ were studied in the earlier literature. For example, using Ruelle-Mayer operator, Faivre \cite{lesFai97} showed that the Lebesgue measure of the set of $x$ for which $k_n(x)/n$ deviates away from $(6\log2\log10)/\pi^2$ decreases exponentially to 0. Later, he also proved a central limit theorem for the sequence $\{k_n, n \geq 1\}$ in \cite{lesFai98}. The law of the iterated logarithm for the sequence $\{k_n, n \geq 1\}$  was established by Wu \cite{lesWu08}. %The main purpose of this paper is to extend the central limit theorem and law of the iterated logarithm for the sequence $\{k_n, n \geq 1\}$ from $\beta =10$ to any $\beta >1$.

We wonder if the similar limit theorems of the sequence $\{k_n, n \geq 1\}$ are still valid for general $\beta >1$. It is worth pointing out that the lengths of cylinders (see Section 2) play an important role in the study of $\beta$-expansion (see \cite{lesB.W14, lesF.W12}). The methods of Faivre (see \cite{lesFai97, lesFai98}) and Wu (see \cite{lesWu06, lesWu08}) rely heavily on the length of a cylinder for $\beta = 10$.
In fact, the cylinder of order $n$ is a regular interval and its length equals always to $10^{-n}$ for the special case $\beta = 10$. For the general case $\beta >1$, it is well-known that the cylinder of order $n$ is a left-closed and right-open interval and its length has an absolute upper bound $\beta^{-n}$. Fan and Wang \cite{lesF.W12} obtained that the growth of the lengths of cylinders is multifractal and that the multifractal spectrum depends on $\beta$. However, for some $\beta >1$, the cylinder of order $n$ is irregular and there is no nontrivial universal lower bound for its length, which can be much smaller than $\beta^{-n}$.
This is the main difficulty we met. We establish a lower bound (not necessarily absolute) of the length of a cylinder (i.e., Proposition \ref{zhongyao1}) to extend the results of Faivre \cite{lesFai98} and Wu \cite{lesWu08} from $\beta =10$ to any $\beta >1$.

Our first result is a central limit theorem for the sequence $\{k_n, n \geq 1\}$, which generalizes the result of Faivre \cite{lesFai98}.

\begin{theorem}\label{central limit theorem}
Let $\beta >1$. For every $y \in \mathbb{R}$,
\[
 \lim_{n \to \infty} \mathrm{P} \left\{x \in [0,1): \frac{k_n(x)-\frac{6\log2\log\beta}{\pi^2}n}{\sigma\sqrt{n}} \leq y \right\} = \frac{1}{\sqrt{2\pi}} \int_{-\infty}^y e^{-\frac{t^2}{2}}dt,
\]
where $\sigma > 0$ is a constant only depending on $\beta$.
\end{theorem}

We also prove a law of the iterated logarithm for the sequence $\{k_n, n \geq 1\}$, which covers the result of Wu \cite{lesWu08}.

\begin{theorem}\label{law of the iterated logarithm}
Let $\beta >1$. For $\mathrm{P}$-almost all $x \in [0,1)$,
  \[
  \limsup\limits_{n \to \infty} \frac{k_n(x)-\frac{6\log2\log\beta}{\pi^2}n}{\sigma\sqrt{2n\log\log n}} = 1
\]
and
\[
  \liminf\limits_{n \to \infty} \frac{k_n(x)-\frac{6\log2\log\beta}{\pi^2}n}{\sigma\sqrt{2n\log\log n}} = -1,
  \]
where $\sigma >0$ is the same constant as in Theorem \ref{central limit theorem}.
\end{theorem}

\section{Preliminary}
This section is devoted to recalling some definitions and basic properties of the $\beta$-expansion and continued fraction expansion. For more properties of $\beta$-expansion and continued fraction expansion, see \cite{lesBla89, lesD.K02, lesF.S92, lesHT08, lesKhi64, lesL.W08, lesPar60, lesRen57, lesSch97, lesSch80} and the references quoted therein.

\subsection{$\beta$-expansion}

 Note that the number 1 is not in the domain of $T_\beta$, but we can still speak of the $\beta$-expansion of 1. Let us introduce the infinite $\beta$-expansion of 1, which is a crucial quantity in the study of $\beta$-expansion. We define $\varepsilon_1(1)= \lfloor\beta\rfloor$ and $\varepsilon_{n +1}(1) = \lfloor\beta T_\beta^n(1)\rfloor$ with $T_\beta (1) = \beta - \lfloor\beta\rfloor$ for all $n \geq 1$. Then the number 1 can be uniquely developed into a finite or infinite series denoted by
\begin{equation*}
1 = \frac{\varepsilon_1(1)}{\beta} + \frac{\varepsilon_2(1)}{\beta^2} + \cdots + \frac{\varepsilon_n(1)}{\beta^n} + \cdots.
\end{equation*}
Here we write the $\beta$-expansion of 1 as $\varepsilon(1, \beta) =(\varepsilon_1(1),\varepsilon_2(1),\cdots, \varepsilon_n(1),\cdots)$. If the $\beta$-expansion of 1 is finite, i.e., $\varepsilon(1, \beta) =(\varepsilon_1(1),\varepsilon_2(1),\cdots, \varepsilon_n(1),0^\infty)$ with $\varepsilon_n(1) \neq 0$, where $\omega^\infty$ denotes the sequence of infinite repetitions of $\omega$, then $\beta$ is called a \emph{simple Parry number}. We define by $(\varepsilon_1^*(1),\varepsilon_2^*(1),\cdots, \varepsilon_n^*(1),\cdots)$ the
infinite $\beta$-expansion of 1 as $(\varepsilon_1^*(1),\varepsilon_2^*(1),\cdots, \varepsilon_n^*(1),\cdots)= \left((\varepsilon_1(1),\varepsilon_2(1),\cdots, \varepsilon_n(1)-1)^\infty\right)$ if $\beta$ is a simple Parry number and as $(\varepsilon_1^*(1),\varepsilon_2^*(1),\cdots, \varepsilon_n^*(1),\cdots) =(\varepsilon_1(1),\varepsilon_2(1),\cdots, \varepsilon_n(1),\cdots)$ if $\beta$ is not a simple Parry number. We sometimes write the infinite $\beta$-expansion of the number 1 as $\varepsilon^*(1, \beta)$ for simplicity.

%In this case, we denote by
%\begin{equation*}
%(\varepsilon_1^*(1),\varepsilon_2^*(1),\cdots, \varepsilon_n^*(1),\cdots) = (\varepsilon_1(1),\varepsilon_2(1),\cdots, \varepsilon_{n-1}(1),\varepsilon_n(1)-1)^\infty.
%\end{equation*}
%If $\beta$ is not a simple Parry number, we denote by
%\begin{equation*}
%(\varepsilon_1^*(1),\varepsilon_2^*(1),\cdots, \varepsilon_n^*(1),\cdots) = %(\varepsilon_1(1),\varepsilon_2(1),\cdots, \varepsilon_n(1),\cdots).
%\end{equation*}
%In both cases, the sequence $(\varepsilon_1^*(1),\varepsilon_2^*(1),\cdots, \varepsilon_n^*(1),\cdots)$  is said to be the infinite $\beta$-expansion of 1 denoted by $\varepsilon^*(1, \beta)$ and we always have that
%\begin{equation*}
%1 = \varepsilon^*(1, \beta) = \frac{\varepsilon_1^*(1)}{\beta} + \frac{\varepsilon_2^*(1)}{\beta^2} + \cdots + \frac{\varepsilon_n^*(1)}{\beta^n} + \cdots.
%\end{equation*}

\begin{definition}
An n-block $(\varepsilon_1 ,\varepsilon_2, \cdots, \varepsilon_n)$ is said to be admissible for $\beta$-expansion if there exists $x \in [0,1)$ such that $\varepsilon_i(x) = \varepsilon_i$ for all $1 \leq i \leq n$. An infinite sequence $(\varepsilon_1 ,\varepsilon_2, \cdots, \varepsilon_n, \cdots)$ is admissible if $(\varepsilon_1 ,\varepsilon_2, \cdots, \varepsilon_n)$ is admissible for all $n \geq 1$.
\end{definition}

 We denote by $\Sigma_\beta^n$ the collection of all admissible sequences with length $n$ and by $\Sigma_\beta$ that of all infinite admissible sequences. The following result of R\'{e}nyi \cite{lesRen57} implies that the dynamical system ([0,1), $T_\beta$) admits $\log \beta$ as its topological entropy.

\begin{proposition}[\cite{lesRen57}]\label{Renyi}
Let $\beta >1$. For any $n \geq 1$,
\begin{equation*}
\beta^n \leq \sharp \Sigma_\beta^n \leq \beta^{n+1}/(\beta -1),
\end{equation*}
where $\sharp$ denotes the number of elements of a finite set.
\end{proposition}

We denote by $\mathcal{A} =\left\{0, 1, \cdots, \lceil\beta\rceil-1\right\}$ the set of digits of the $\beta$-expansion, where $\lceil x\rceil$ denotes the smallest integer no less than $x$.
Let $\mathcal{W} = \mathcal{A}^\mathbb{N}$ be the symbolic space with the one-sided shift $\theta$ and the lexicographical ordering $\prec$, that is, $(\varepsilon_1 ,\varepsilon_2, \cdots, \varepsilon_n, \cdots) \prec (\varepsilon_1^{\prime} ,\varepsilon_2^{\prime}, \cdots, \varepsilon_n^{\prime}, \cdots)$ means that there exists $k \geq 1$ such that $\varepsilon_i = \varepsilon_i^{\prime}$ for all $1 \leq i < k$ and $\varepsilon_k < \varepsilon_k^{\prime}$.
It is well-known that not all sequences in $\mathcal{W}$ belong to $\Sigma_\beta$ when $\beta$ is not an integer (see \cite[Section 3.3]{lesD.K02}). The following proposition, due to Parry \cite{lesPar60}, gives a characterization of all admissible sequences which relies heavily on the infinite $\beta$-expansion of the number 1.

\begin{proposition}[\cite{lesPar60}]\label{parry pr}
Let $\varepsilon^*(1, \beta)$ be the infinite $\beta$-expansion of 1 and $\omega \in \mathcal{W}$. Then $\omega \in \Sigma_\beta$ if and only if
\[
\theta^n(\omega) \prec \varepsilon^*(1, \beta)\ \  \text{for all} \ \ n \geq 0.
\]
\end{proposition}

\begin{definition}
Let $(\varepsilon_1 ,\varepsilon_2, \cdots, \varepsilon_n)\in \Sigma_\beta^n$. We define
\[
J(\varepsilon_1 ,\varepsilon_2, \cdots, \varepsilon_n) = \{x \in [0,1):\varepsilon_i(x)=\varepsilon_i \ \text{for all} \ 1 \leq i \leq n\}
\]
and call it the cylinder of order $n$ of $\beta$-expansion, i.e., it is the set of points whose $\beta$-expansion starts with $(\varepsilon_1 ,\varepsilon_2, \cdots, \varepsilon_n)$. For any $x \in [0,1)$, $J(\varepsilon_1(x) ,\varepsilon_2(x), \cdots, \varepsilon_n(x))$ is said to be the cylinder of order $n$ containing $x$.
\end{definition}

Let $(\varepsilon_1 ,\varepsilon_2, \cdots, \varepsilon_n)\in \Sigma_\beta^n$. As we know, $J(\varepsilon_1 ,\varepsilon_2, \cdots, \varepsilon_n)$ is a left-closed and right-open interval with left endpoint
\begin{equation*}
\frac{\varepsilon_1}{\beta} + \frac{\varepsilon_2}{\beta^2} + \cdots + \frac{\varepsilon_n}{\beta^n}.
\end{equation*}
Moreover, the length of $J(\varepsilon_1 ,\varepsilon_2, \cdots, \varepsilon_n)$ satisfies
$|J(\varepsilon_1 ,\varepsilon_2, \cdots, \varepsilon_n)| \leq 1/\beta^n$. For any $x \in [0,1)$ and $n \geq 1$, we
assume that $(\varepsilon_1(x), \varepsilon_2(x), \cdots, \varepsilon_n(x), \cdots)$ is the $\beta$-expansion of $x$ and define
\begin{equation}\label{l-n}
l_n(x) = \sup\left\{k \geq 0: \varepsilon_{n +j}(x) =0 \ \text{for all}\ 1 \leq j \leq k\right\}.
\end{equation}
That is, the length of the longest string of zeros just after the $n$-th digit in the $\beta$-expansion of $x$.

\begin{proposition}\label{zhongyao1}
Let $\beta >1$. Then for any $x \in [0,1)$ and $n \geq 1$,
\begin{equation*}
\frac{1}{\beta^{n+l_n(x)+1}}\leq |J(\varepsilon_1(x) ,\varepsilon_2(x), \cdots, \varepsilon_n(x))| \leq \frac{1}{\beta^n},
\end{equation*}
where $l_n(x)$ is defined as (\ref{l-n}).
\end{proposition}

\begin{proof}
For any $x \in [0,1)$ and $n \geq 1$, we know that $J(\varepsilon_1(x) ,\varepsilon_2(x), \cdots, \varepsilon_n(x))$ is a left-closed and right-open interval with left endpoint
\begin{equation*}
\omega_n(x):= \frac{\varepsilon_1(x)}{\beta} + \frac{\varepsilon_2(x)}{\beta^2} + \cdots + \frac{\varepsilon_n(x)}{\beta^n}
\end{equation*}
and its length satisfies $|J(\varepsilon_1(x) ,\varepsilon_2(x), \cdots, \varepsilon_n(x))| \leq 1/\beta^n.$

Since $x \in J(\varepsilon_1(x) ,\varepsilon_2(x), \cdots, \varepsilon_n(x))$, we have that
\begin{equation*}
\frac{1}{\beta^{n+l_n(x)+1}} \leq x - \omega_n(x) \leq |J(\varepsilon_1(x) ,\varepsilon_2(x), \cdots, \varepsilon_n(x))|,
\end{equation*}
where the first inequality follows from $\varepsilon_{n+1}(x) = \cdots = \varepsilon_{n+l_n(x)}(x) = 0$ and $\varepsilon_{n+l_n(x)+1}(x) \geq 1$ by the definition of $l_n(x)$ in (\ref{l-n}). This completes the proof.
\end{proof}

\subsection{Continued fraction expansion}
With the conventions $p_{-1}=1$, $q_{-1}=0$, $p_0=0$, $q_0=1$, the quantities $p_n$ and $q_n$ satisfy the following recursive formula.

\begin{proposition}[\cite{lesKhi64}]\label{recursive}
For any irrational number $x \in [0,1)$ and $n \geq 1$,
\begin{equation*}
p_n(x) = a_n(x) p_{n-1}(x) + p_{n-2}(x)\ \ \ \text{and}\ \ \ q_n(x) = a_n(x) q_{n-1}(x) + q_{n-2}(x).
\end{equation*}
\end{proposition}

Note that $q_n(x) = a_n(x) q_{n-1}(x) + q_{n-2}(x) \geq 2 q_{n-2}(x)$ for any irrational number $x \in [0,1)$ and $n \geq 1$, we obtain immediately the following corollary as an application of Proposition \ref{recursive}.

\begin{corollary}\label{cor}
For any irrational number $x \in [0,1)$ and $n \geq 1$,
\begin{equation*}
q_n(x) \geq 2^{\frac{n-k-1}{2}}q_k(x)\ \ \ \text{for all}\ \ \ 0 \leq k < n.
\end{equation*}
\end{corollary}

It is known that the Gauss transformation $T$ does not preserve the Lebesgue measure $\mathrm{P}$, but there exists a $T$-invariant measure $\nu$ equivalent to the Lebesgue measure $\mathrm{P}$, namely the Gauss measure $\nu$ defined by
\[
\nu(A) = \frac{1}{\log 2}\int_A \frac{1}{1+x}dx
\]
for any Borel set $A \subseteq [0, 1)$. More precisely,
\begin{equation}\label{eqivalent}
 \frac{1}{2\log2}\mathrm{P}(A) \leq \nu(A) \leq \frac{1}{\log2}\mathrm{P}(A)
\end{equation}
for any Borel set $A \subseteq [0, 1)$. Thus if a certain property holds for $\nu$-almost all $x \in [0,1)$, then it also holds for $\mathrm{P}$-almost all $x \in [0,1)$. Hence from this point of view, it makes no difference which of them we use. Philipp \cite{lesPhi70} showed that the sequence of partial quotients $\{a_n, n \geq 1\}$ is a $\psi$-mixing stationary  process  w.r.t.~the Gauss measure $\nu$. Recall that the random variables sequence $\{a_n, n \geq 1\}$ is $\psi$-mixing w.r.t.~$\nu$ if its $\psi$-mixing cofficients
\[
\psi(n) = \sup\left|\frac{\nu(A\cap B)}{\nu(A)\nu(B)} -1\right|
\]
goes to zero as $n$ tends to $\infty$, where the supremum is taken over all $A \in \mathcal{B}^k_1$ and $B \in \mathcal{B}^\infty_{k+n}$ such that $\nu(A)\nu(B)>0~(k \in \mathbb{N})$ and the quantities $\mathcal{B}^k_1$ and $\mathcal{B}^\infty_{k+n}$ denote the $\sigma$-algebras generated by the random variables $a_1,\cdots,a_k$ and $a_{k+n},a_{k+n+1},\cdots$ respectively (see the survey paper by Bradley \cite{lesBra05}).

The following proposition establishes a relation between $q_n(x)$ and the orbit of $x$ under the Gauss transformation $T$.

\begin{proposition}\label{jie}
For any irrational number $x \in [0,1)$ and $n \geq 1$,
\begin{equation*}
|\log q_n(x)+(\log x + \cdots + \log T^{n-1}x)| \leq \log2.
\end{equation*}
\end{proposition}

\begin{proof}
For any irrational number $x \in [0,1)$ and $n \geq 1$, the form (\ref{continued fraction expansion}) and Proposition \ref{recursive} yield that
\[
x = \frac{p_n(x) + T^nx p_{n-1}(x)}{q_n(x) + T^nx q_{n-1}(x)}.
\]
That is,
\[
T^nx = - \frac{xq_n(x)- p_n(x)}{xq_{n-1}(x) - p_{n-1}(x)}.
\]
Hence that
\[
x \cdot Tx \cdot \cdots \cdot T^{n-1}x = (-1)^{n-2} x \cdot \frac{xq_1 - p_1}{xq_0 - p_0} \cdot \cdots \cdot \frac{xq_{n-1}(x) - p_{n-1}(x)}{xq_{n-2}(x) - p_{n-2}(x)} = |xq_{n-1}(x) - p_{n-1}(x)|.
\]

Note that the inequalities (\ref{diophantine}) show that
\[
 \frac{1}{2q_n(x)q_{n-1}(x)} \leq \left|x-\frac{p_{n-1}(x)}{q_{n-1}(x)}\right| \leq \frac{1}{q_n(x)q_{n-1}(x)},
\]
thus
\[
\frac{1}{2q_n(x)} \leq x \cdot Tx \cdot \cdots \cdot T^{n-1}x \leq \frac{1}{q_n(x)}.
\]
Taking the logarithm on both sides of the above inequalities, we obtain the desired result.
\end{proof}

The well-known theorem of  L\'{e}vy \cite{lesLev29} about the growth of $q_n(x)$ is deduced immediately by Proposition \ref{jie} and Birkhoff's ergodic theorem (see \cite[Theorem 3.1.7]{lesD.K02}). That is, for $\mathrm{P}$-almost all $x \in [0,1)$,
\[
\lim\limits_{n \to \infty}\frac{1}{n}\log q_n(x) = - \lim\limits_{n \to \infty} \frac{1}{n}(\log x + \cdots + \log T^{n-1}x) = -\int_0^1 \log x d\nu(x) = \frac{\pi^2}{12\log2}.
\]

\begin{definition}\label{cylinder}
For any $n \geq 1$ and $a_1, a_2, \cdots, a_n \in \mathbb{N}$, we call
\begin{equation*}
I(a_1, \cdots, a_n):= \left\{x \in [0,1): a_1(x)=a_1, \cdots, a_n(x)=a_n\right\}
\end{equation*}
the cylinder of order $n$ of continued fraction expansion. In other words, it is the set of points beginning with $(a_1,\cdots, a_n)$ in their continued fraction expansions. For any irrational $x \in [0,1)$, $I(a_1(x) ,a_2(x), \cdots, a_n(x))$ is said to be the cylinder of order $n$ containing $x$.
\end{definition}

The following proposition is about the structure and length of a cylinder.

\begin{proposition}[\cite{lesD.K02}]
Let $a_1, a_2, \cdots, a_n \in \mathbb{N}$. Then $I(a_1, \cdots, a_n)$ is a half-open and half-closed interval with two endpoints
\[
\frac{p_n}{q_n}\ \ \ \ \ \text{and}\ \ \ \ \  \frac{p_n+p_{n-1}}{q_n+q_{n-1}}
\]
and the length of $I(a_1, \cdots, a_n)$ satisfies
\[
|I(a_1, \cdots, a_n)| = \frac{1}{q_n(q_n+q_{n-1})},
\]
where $p_n$ and $q_n$ satisfy the recursive formula in Proposition \ref{recursive}.
\end{proposition}

For any irrational $x \in [0,1)$ and $n \geq 1$, recall that
\[
k_n(x) = \sup\left\{m \geq 0: J(\varepsilon_1(x),\cdots, \varepsilon_n(x)) \subset I(a_1(x), \cdots, a_m(x))\right\}.
\]
The following connection between $q_n(x)$ and $k_n(x)$ was established by Li and Wu \cite{lesL.W08}, which plays an important role in studying the relationship between $\beta$-expansion and continued fraction expansion.

\begin{proposition}[\cite{lesL.W08}]\label{fundamental 2}
Let $x \in [0,1)$ be an irrational number. Then for any $n \geq 1$,
\[
\frac{1}{6q_{k_n(x)+3}^2(x)} \leq |J(\varepsilon_1(x), \varepsilon_2(x), \cdots, \varepsilon_n(x))| \leq \frac{1}{q_{k_n(x)}^2(x)}.
\]
\end{proposition}

\section{Proof of Theorems}
Denote the constants $$ a = \frac{6\log2 \log\beta}{\pi^2}\ \ \ \ \text{and} \ \ \ \  b= \frac{\pi^2}{12\log2}.$$
We use the notation $\mathrm{E}(\xi)$ to denote the expectation of a random variable $\xi$ w.r.t. $\mathrm{P}$.

\subsection{Proof of Theorem \ref{central limit theorem}}

The following theorem called Cram\'{e}r's theorem or Slutsky's theorem will be used in our proofs (see \cite[Theorem 5.11.4]{lesGut05}).

\begin{theorem}[\cite{lesGut05}]\label{Slutsky Theorem}
Let $\{X_n, n \geq 1\}$, $\{Y_n, n \geq 1\}$, $\{Z_n, n \geq 1\}$ be the sequences of real-valued random variables satisfying
\begin{equation*}
X_n \stackrel{d}{\longrightarrow} X,\ \ \ \ \  Y_n \stackrel{a.s.}{\longrightarrow} 1\ \ \ \ \ \text{and}\ \ \ \ \  Z_n \stackrel{p}{\longrightarrow} 0\ \ \ \ \ \text{as}\ \ n \to \infty,
\end{equation*}
where $\stackrel{d}{\longrightarrow}$, $\stackrel{a.s.}{\longrightarrow}$ and $\stackrel{p}{\longrightarrow}$ denote ``convergence in distribution", ``convergence almost surely" and ``convergence in probability" respectively. Then
\begin{equation*}
X_nY_n+Z_n \stackrel{d}{\longrightarrow} X\ \ \ \ \text{as} \ \ n \to \infty.
\end{equation*}
\end{theorem}

The following lemma shows that in Theorem \ref{central limit theorem for qn} we can replace the limit on $\{n\}$ by the subsequence $\{k_n(x)\}$ which depends on both $x$ and $\beta$.

\begin{lemma}\label{central limit theorem zilie}
Let $\beta >1$. For every $y \in \mathbb{R}$,
\begin{align*}
\lim_{n \to \infty} \mathrm{P} \left\{x \in [0,1): \frac{\log q_{k_n(x)}(x)-bk_n(x)}{\sigma_1\sqrt{k_n(x)}} \leq y \right\} = \frac{1}{\sqrt{2\pi}} \int_{-\infty}^y e^{-\frac{t^2}{2}}dt.
\end{align*}
\end{lemma}

\begin{proof}
For any irrational $x \in [0,1)$ and $n \geq 1$, denote
\[
A_n(x) = \log q_n(x) - bn\ \ \ \  \text{and}\ \ \ \ S_n(x) = Y_1(x) + \cdots + Y_n(x),
\]
where $T^0x =x$ and $Y_k(x) = -\log T^{k-1}x -b$ for all $1 \leq k \leq n$. By Proposition \ref{jie} and Theorem \ref{central limit theorem for qn}, we deduce that
\begin{equation}\label{S-n clt}
\frac{S_n}{\sigma_1\sqrt{n}} \Longrightarrow N(0,1),
\end{equation}
where $X_n  \Longrightarrow N(\mu, \sigma)$ denotes that $X_n \stackrel{d}{\longrightarrow} X$ and $X$ is a normal distribution random variable with mean $\mu$ and variance $\sigma^2$.

Notice that
\[
\frac{A_{k_n}}{\sigma_1\sqrt{k_n}} = \frac{A_{k_n}}{\sqrt{n}}\cdot\frac{\sqrt{n}}{\sigma_1\sqrt{k_n}}
\]
and $\lim\limits_{n \to \infty} k_n(x)/n = a$ for $\mathrm{P}$-almost all $x \in [0,1)$, to prove Lemma \ref{central limit theorem zilie}, it suffices to show that
\[
\frac{A_{k_n}}{\sqrt{n}} \Longrightarrow N(0,a\sigma_1^2).
\]
By Proposition \ref{jie}, it is equivalent to prove that
\begin{equation}\label{equivalent}
\frac{S_{k_n}}{\sqrt{n}} \Longrightarrow N(0,a\sigma_1^2).
\end{equation}

Let $t_n = \lfloor an\rfloor+1$. In view of (\ref{theorem eq}), we have that $\lim\limits_{n \to \infty} k_n(x)/t_n = 1$ for $\mathrm{P}$-almost all $x \in [0,1)$. By (\ref{S-n clt}), the assertion (\ref{equivalent}) holds if we can show that
\begin{equation}\label{ruoshoulian}
\frac{S_{k_n} - S_{t_n}}{\sqrt{n}} \stackrel{p}{\longrightarrow} 0.
\end{equation}
To do this, we need to use the equation (22) of Faivre \cite{lesFai98}, which states that for any $\lambda >0$,
\begin{equation}\label{faivre theorem1}
\limsup_{n \to \infty} \mathrm{P}\left(\max_{1 \leq i \leq n} |S_i| \geq \lambda \sqrt{n}\right) \leq \frac{16K}{\lambda^2},
\end{equation}
where $K = \sup_{n \geq 1}\frac{\mathrm{E}(S_n^2)}{n} < \infty$ is an absolute constant.

Now we use (\ref{faivre theorem1}) to prove the assertion (\ref{ruoshoulian}). For any $\varepsilon >0$ and any $0 < \delta <1$,
\begin{align}\label{faivre 1}
\mathrm{P}\left(|S_{k_n} - S_{t_n}|\geq \varepsilon \sqrt{n}\right) &\leq \mathrm{P}\left(|S_{k_n} - S_{t_n}| \geq \varepsilon \sqrt{n}, \left|\frac{k_n}{t_n}-1\right| \leq \delta\right) + \mathrm{P}\left(\left|\frac{k_n}{t_n}-1\right| > \delta\right) \notag \\
&\leq \mathrm{P}\left(\max_{(1-\delta)t_n \leq i \leq (1+\delta)t_n} |S_i - S_{t_n}| \geq \varepsilon \sqrt{n}\right) + \mathrm{P}\left(\left|\frac{k_n}{t_n}-1\right| > \delta\right).
\end{align}
Note that $\lim\limits_{n \to \infty} k_n(x)/t_n = 1$ for $\mathrm{P}$-almost all $x \in [0,1)$, in particular, $\lim\limits_{n \to \infty} k_n(x)/t_n = 1$ in probability. Then the probability $\mathrm{P}(|k_n/t_n -1| > \delta) \to 0$ as $n \to \infty$. Since the stochastic process $\{Y_n, n \geq 1\}$ is stationary w.r.t. the Gauss measure $\nu$, we obtain that
\begin{equation*}
\nu\left(\max_{(1-\delta)t_n \leq i \leq (1+\delta)t_n} |S_i - S_{t_n}| \geq \varepsilon \sqrt{n}\right) \leq
2\nu\left(\max_{1 \leq i \leq \lfloor\delta t_n\rfloor} |S_i| \geq \varepsilon \sqrt{n}\right).
\end{equation*}
Hence the inequalities (\ref{eqivalent}) imply that
\begin{equation}\label{faivre 2}
\mathrm{P}\left(\max_{(1-\delta)t_n \leq i \leq (1+\delta)t_n} |S_i - S_{t_n}| \geq \varepsilon \sqrt{n}\right) \leq
4\mathrm{P}\left(\max_{1 \leq i \leq \lfloor\delta t_n\rfloor} |S_i| \geq \varepsilon \sqrt{n}\right).
\end{equation}

For sufficient large $n$ such that $\sqrt{n} \geq \frac{\sqrt{\lfloor\delta t_n\rfloor}}{2\sqrt{a\delta}}$, we have that
\begin{equation}\label{faivre 3}
\mathrm{P}\left(\max_{1 \leq i \leq \lfloor\delta t_n\rfloor} |S_i| \geq \varepsilon \sqrt{n}\right) \leq \mathrm{P}\left(\max_{1 \leq i \leq \lfloor\delta t_n\rfloor} |S_i| \geq \frac{\varepsilon}{2\sqrt{a\delta}} \sqrt{\lfloor\delta t_n\rfloor}\right).
\end{equation}
Since $\{\lfloor\delta t_n\rfloor\}$ is a subsequence of \{n\}, applying (\ref{faivre theorem1}) with $\lambda = \frac{\varepsilon}{2\sqrt{a\delta}}$, we deduce
\begin{equation}\label{faivre 4}
\limsup_{n \to \infty} \mathrm{P}\left(\max_{1 \leq i \leq \lfloor\delta t_n\rfloor} |S_i| \geq \frac{\varepsilon}{2\sqrt{a\delta}} \sqrt{\lfloor\delta t_n\rfloor}\right) \leq \frac{64Ka\delta}{\varepsilon^2}.
\end{equation}

Combining (\ref{faivre 1}), (\ref{faivre 2}), (\ref{faivre 3}) and (\ref{faivre 4}), we obtain that
\begin{equation*}
\limsup_{n \to \infty} \mathrm{P}\left(|S_{k_n} - S_{t_n}|\geq \varepsilon \sqrt{n}\right) \leq \frac{256Ka\delta}{\varepsilon^2},
\end{equation*}
which implies
\begin{equation*}
\lim_{n \to \infty} \mathrm{P}\left(|S_{k_n} - S_{t_n}|\geq \varepsilon \sqrt{n}\right) =0,
\end{equation*}
by the arbitrariness of $0<\delta<1$.
Therefore, (\ref{ruoshoulian}) holds.
\end{proof}

\begin{lemma}\label{a-n shang}
Let $m \in \mathbb{N}$. Then for $\mathrm{P}$-almost all $x \in [0,1)$,
\begin{equation}\label{a-n jixian}
\lim_{n \to \infty}\frac{\log a_{k_n(x)+m}(x)}{\sqrt{n}} =0.
\end{equation}
\end{lemma}

\begin{proof}
Let $B$ be the set of irrational numbers such that the limit in (\ref{a-n jixian}) does not converge to 0. For any $x \in B$, we have that
\begin{align*}
\limsup_{n \to \infty}\frac{\log a_{k_n(x)+m}(x)}{\sqrt{n}} >0.
\end{align*}
Note that (\ref{theorem eq}) holds for $\mathrm{P}$-almost all $x \in [0,1)$. Thus we can assume that $k_n(x)/n \to a$ as $n \to \infty$ since (\ref{a-n jixian}) is a kind of almost all result. Therefore, we obtain
\[
\limsup_{n \to \infty}\frac{\log a_{k_n(x)+m}(x)}{\sqrt{k_n(x)+m}} >0.
\]
In other words, there exists $\varepsilon_0 >0$ such that
\begin{equation}\label{infinite}
a_{k_n(x)+m}(x) \geq e^{\varepsilon_0 \sqrt{k_n(x)+m}} \ \ \text{holds for infinite many $n$}.
\end{equation}
Hence $a_n(x) \geq e^{\varepsilon_0 \sqrt{n}}$ are satisfied for infinite many $n$ since $k_n(x) \to \infty$ as $n \to \infty$. However, from the series $\sum_{n \geq 1} e^{-\varepsilon_0 \sqrt{n}} < + \infty$ and the Borel-Bernstein 0-1 law (see \cite[Theorem 30]{lesKhi64}) for $\{a_n, n\geq 1\}$, we know that for $\mathrm{P}$-almost all $y \in [0,1)$, the inequalities $a_n(y) \geq e^{\varepsilon_0 \sqrt{n}}$ are satisfied for only finitely many $n$. Thus, the set of points satisfying (\ref{infinite}) is zero of Lebesgue measure. Hence that $\mathrm{P}(B) = 0$.
\end{proof}

\begin{lemma}\label{l-n genhao}
Let $\beta >1$. Then for $\mathrm{P}$-almost all $x \in [0,1)$,
\begin{align*}
\lim_{n \to \infty} \frac{l_n(x)}{\sqrt{n}} =0.
\end{align*}
\end{lemma}

\begin{proof}
By the definition of $l_n(x)$ in (\ref{l-n}), it is clear that $\liminf\limits_{n \to \infty} l_n(x)/\sqrt{n} \geq 0$. To prove
$\limsup\limits_{n \to \infty} l_n(x)/\sqrt{n} \leq 0$ for $\mathrm{P}$-almost all $x \in [0,1)$, we only need to show for any $\varepsilon > 0$,
\begin{equation}\label{borel}
\sum_{n=1}^{\infty} \mathrm{P}\left\{x \in [0,1): l_n(x) \geq \varepsilon\sqrt{n}\right\} < + \infty,
\end{equation}
since by the Borel-Cantelli lemma (see \cite[Theorem 2.18.1]{lesGut05}), the formula (\ref{borel}) implies that
\begin{equation}\label{cite}
\mathrm{P} \left\{x \in [0,1): \limsup_{n \to \infty} \frac{l_n(x)}{\sqrt{n}} \leq \varepsilon\right\} =1.
\end{equation}
We denote by $A= \bigcap\limits_{k=1}^\infty \left\{x \in [0,1): \limsup\limits_{n \to \infty} l_n(x)/\sqrt{n} \leq  1/k\right\}$. Then for any $k \geq 1$, applying (\ref{cite}) with $\varepsilon = 1/k$, we have 
$\mathrm{P} \left\{x \in [0,1): \limsup\limits_{n \to \infty} l_n(x)/\sqrt{n} \leq  1/k \right\} =1$ and therefore $\mathrm{P}(A) =1$. For any $x \in A$, we obtain that $\limsup\limits_{n \to \infty} l_n(x)/\sqrt{n} \leq  1/k$ holds for all $k \geq 1$. Thus, $\limsup\limits_{n \to \infty} l_n(x)/\sqrt{n} \leq 0$ holds for $\mathrm{P}$-almost all $x \in [0,1)$.

%For any $k \geq 1$, applying $\varepsilon = 1/k$, we have that
%$\mathrm{P} \left\{x \in [0,1): \limsup\limits_{n \to \infty} l_n(x)/\sqrt{n} \leq  1/k \right\} =1.$
%Hence that
%\begin{equation*}
%\mathrm{P}\left(\bigcap_{k=1}^\infty \left\{x \in [0,1): \limsup_{n \to \infty} \frac{l_n(x)}{\sqrt{n}} \leq  \frac{1}{k}\right\} \right)=1.
%\end{equation*}
%Let $A= \bigcap\limits_{k=1}^\infty \left\{x \in [0,1): \limsup\limits_{n \to \infty} \frac{l_n(x)}{\sqrt{n}} \leq  \frac{1}{k}\right\}$. Then $\mathrm{P}(A) =1$. For any $x \in A$, we obtain that $\limsup\limits_{n \to \infty} l_n(x)/\sqrt{n} \leq  1/k$ holds for all $k \geq 1$. Thus, $\limsup\limits_{n \to \infty} l_n(x)/\sqrt{n} \leq 0$ holds for $\mathrm{P}$-almost all $x \in [0,1)$.

Now we prove (\ref{borel}). In fact, the set $A_n(\varepsilon):= \{x \in [0,1): l_n(x) \geq \varepsilon\sqrt{n}\}$ is a subset of the union of the cylinders of order  $(n+s_n)$ like $J(\varepsilon_1,\cdots,\varepsilon_n, \underbrace{0,\cdots,0}_{s_n})$ with $s_n = \lfloor\varepsilon\sqrt{n}\rfloor$, where $(\varepsilon_1,\cdots,\varepsilon_n) \in \Sigma_{\beta}^n$. That is,
\[
A_n(\varepsilon) \subseteq \bigcup_{(\varepsilon_1,\cdots,\varepsilon_n) \in \Sigma_{\beta}^n} J(\varepsilon_1,\cdots,\varepsilon_n, \underbrace{0,\cdots,0}_{s_n}).
\]
Since $|J(\varepsilon_1,\cdots,\varepsilon_n, \underbrace{0,\cdots,0}_{s_n})| \leq 1/\beta^{n+s_n}$ for any $(\varepsilon_1,\cdots,\varepsilon_n) \in \Sigma_{\beta}^n$, Proposition \ref{Renyi} implies that
\[
\mathrm{P}(A_n(\varepsilon)) \leq \sum_{(\varepsilon_1,\cdots,\varepsilon_n) \in \Sigma_{\beta}^n} |J(\varepsilon_1,\cdots,\varepsilon_n, \underbrace{0,\cdots,0}_{s_n})| \leq \frac{\beta^{n+1}}{\beta - 1} \cdot \frac{1}{\beta^{n+s_n}} \leq \frac{\beta^2}{\beta - 1} \cdot \frac{1}{\beta^{\varepsilon\sqrt{n}}}.
\]
Notice that $\sum_{n \geq 1}\beta^{-\varepsilon\sqrt{n}} < + \infty$, so we get the desired result.
\end{proof}

\begin{lemma}\label{q-n jixian}
Let $\beta >1$. Then for $\mathrm{P}$-almost all $x \in [0,1)$,
\begin{align*}
\lim_{n \to \infty} \frac{\log q_{k_n(x)}(x)-abn}{\sqrt{n}} =0.
\end{align*}
\end{lemma}

\begin{proof}
For any irrational $x \in [0,1)$ and $n \geq 1$, denote
\begin{equation*}
W_n(x)= \frac{\log q_{k_n(x)}(x)-abn}{\sqrt{n}} \ \ \text{and}\ \  W_n^{\prime}(x) =\frac{\log q_{k_n(x)+3}(x)-abn}{\sqrt{n}}.
\end{equation*}
We will show that for $\mathrm{P}$-almost all $x \in [0,1)$,
\begin{equation}\label{jixian shang}
\limsup_{n \to \infty} W_n(x) \leq 0 \ \ \text{and}\ \  \liminf_{n \to \infty} W_n^{\prime}(x) \geq 0.
\end{equation}

Propositions \ref{zhongyao1} and \ref{fundamental 2} show respectively that
\begin{equation}\label{budengshi1}
\frac{1}{\beta^{n +l_n(x) +1}}\leq |J(\varepsilon_1(x) , \cdots, \varepsilon_n(x))| \leq \frac{1}{\beta^n}
\end{equation}
and
\begin{equation}\label{budengshi2}
\frac{1}{6q_{k_n(x)+3}^2(x)} \leq |J(\varepsilon_1(x) , \cdots, \varepsilon_n(x))| \leq \frac{1}{q_{k_n(x)}^2(x)}.
\end{equation}
In view of the right inequality of (\ref{budengshi1}) and the left inequality of (\ref{budengshi2}), we deduce that
\begin{equation*}
\log q_{k_n(x)+3}(x) -\frac{\log \beta}{2}n \geq - \frac{\log6}{2}.
\end{equation*}
Hence $\liminf\limits_{n \to \infty} W_n^{\prime}(x) \geq 0$. Next we are ready to prove $\limsup\limits_{n \to \infty} W_n(x) \leq 0$ for $\mathrm{P}$-almost all $x \in [0,1)$. Combining the left inequality of (\ref{budengshi1}) and the right inequality of (\ref{budengshi2}), we obtain that
\begin{equation*}
\log q_{k_n(x)}(x) -\frac{\log \beta}{2}n \leq (l_n(x) +1)\frac{\log\beta}{2}.
\end{equation*}
Therefore, Lemma \ref{l-n genhao} implies that $\limsup\limits_{n \to \infty} W_n(x) \leq 0$ for $\mathrm{P}$-almost all $x \in [0,1)$.

By Proposition \ref{recursive}, the recursive formula of $q_n$ shows that
\begin{equation*}
q_{k_n(x)+3} \leq 8 a_{k_n(x)+3}a_{k_n(x)+2}a_{k_n(x)+1}q_{k_n(x)}.
\end{equation*}
So, we have that
\begin{align}\label{lem}
W_n(x) \geq W_n^{\prime}(x)- \frac{\log8}{\sqrt{n}} - \sum_{m=1}^3\frac{\log a_{k_n(x)+m}(x)}{\sqrt{n}}.
\end{align}
Applying $m = 1,2,3$ to Lemma \ref{a-n shang}, we obtain that for $\mathrm{P}$-almost all $x \in [0,1)$,
\begin{equation*}
\lim_{n \to \infty}\sum_{m=1}^3\frac{\log a_{k_n(x)+m}(x)}{\sqrt{n}} = 0.
\end{equation*}
By (\ref{lem}), we deduce that
\begin{equation}\label{jixian xia}
\liminf_{n \to \infty} W_n(x) \geq \liminf_{n \to \infty} W_n^{\prime}(x) - \lim_{n \to \infty} \frac{\log8}{\sqrt{n}} - \lim_{n \to \infty}\sum_{m=1}^3\frac{\log a_{k_n(x)+m}(x)}{\sqrt{n}} \geq 0.
\end{equation}

Combing the first inequality of (\ref{jixian shang}) and (\ref{jixian xia}), we complete the proof.
\end{proof}

Now we are going to prove Theorem \ref{central limit theorem}.

\begin{proof}[Proof of Theorem \ref{central limit theorem}]
For any irrational $x \in [0,1)$ and $n \geq 1$, let
\begin{align*}
X_n(x) = - \frac{\log q_{k_n(x)}(x)-bk_n(x)}{\sigma_1\sqrt{k_n(x)}},
\end{align*}
\begin{align*}
Y_n(x)= \frac{\sigma_1\sqrt{k_n(x)}}{b\sigma\sqrt{n}}\ \ \ \ \ \ \text{and}\ \ \ \ \ \  Z_n(x)= \frac{\log q_{k_n(x)}(x)-abn}{b\sigma\sqrt{n}}.
\end{align*}

Therefore,
\begin{equation*}
\frac{k_n(x)-an}{\sigma\sqrt{n}} = X_n(x)\cdot Y_n(x) + Z_n(x),
\end{equation*}
where $\sigma$ and $\sigma_1$ are related by the equation
\begin{equation}\label{sigma constant}
\sigma^2 =\frac{a}{b^2}\sigma_1^2 =\frac{864\log^3 2 \log\beta}{\pi^6}\sigma_1^2.
\end{equation}
By Lemma \ref{central limit theorem zilie}, the sequence $\{X_n, n\geq 1\}$ converges to the standard normal distribution in distribution. The equalities (\ref{theorem eq}) and (\ref{sigma constant}) guarantee that $Y_n(x) \to 1$ as $n \to \infty$ for $\mathrm{P}$-almost all $x \in [0,1)$. From Lemma \ref{q-n jixian}, we know that for $\mathrm{P}$-almost all $x \in [0,1)$, $Z_n(x) \to 0$ as $n \to \infty$. In particular, $\lim\limits_{n \to \infty} Z_n = 0$ in probability. Therefore, Theorem \ref{Slutsky Theorem} implies that the sequence $\{X_nY_n + Z_n, n\geq 1\}$ converges to the standard normal distribution in distribution. That is, for every $y \in \mathbb{R}$,
\[
 \lim_{n \to \infty} \mathrm{P} \left\{x \in [0,1): \frac{k_n(x)-\frac{6\log2\log\beta}{\pi^2}n}{\sigma\sqrt{n}} \leq y \right\} = \frac{1}{\sqrt{2\pi}} \int_{-\infty}^y e^{-\frac{t^2}{2}}dt.
\]
\end{proof}

\subsection{Proof of Theorem \ref{law of the iterated logarithm}}

By Lemmas \ref{l-n genhao} and \ref{q-n jixian}, we immediately obtain the following result.

\begin{lemma}\label{l-n chongduishu}
Let $\beta >1$. Then for $\mathrm{P}$-almost all $x \in [0,1)$,
\begin{align*}
\lim_{n \to \infty} \frac{l_n(x)}{\sqrt{n\log\log n}} =0\ \ \ \ \ \ \text{and}\ \ \ \ \ \ \lim_{n \to \infty} \frac{\log q_{k_n(x)}(x)-abn}{\sqrt{n\log\log n}} =0.
\end{align*}
\end{lemma}

\begin{lemma}\label{inequality 2}
Let $\beta >1$. Then for $\mathrm{P}$-almost all $x \in [0,1)$,
\begin{equation*}
\lim\limits_{n \to \infty}\frac{k_{n+1}(x)-k_n(x)}{\sqrt{n\log\log n}} = 0.
\end{equation*}
\end{lemma}

\begin{proof}
By (\ref{increasing sequence}), we know that $\liminf\limits_{n\to \infty}(k_{n+1}(x)-k_n(x)) \geq 0$ for all irrational $x \in [0,1)$. So it suffices to prove that for $\mathrm{P}$-almost all $x \in [0,1)$,
\[
\limsup\limits_{n \to \infty}\frac{k_{n+1}(x)-k_n(x)}{\sqrt{n\log\log n}} \leq 0.
\]

For any irrational $x \in [0,1)$ and $n \geq 1$, Propositions \ref{zhongyao1} and \ref{fundamental 2} show respectively that
\begin{equation*}
\frac{1}{\beta^{n +l_n(x) +1}}\leq |J(\varepsilon_1(x) , \cdots, \varepsilon_n(x))| \leq \frac{1}{\beta^n}
\end{equation*}
and
\begin{equation*}
\frac{1}{6q_{k_n(x)+3}^2(x)} \leq |J(\varepsilon_1(x) , \cdots, \varepsilon_n(x))| \leq \frac{1}{q_{k_n(x)}^2(x)}.
\end{equation*}
Therefore, we have that
\begin{equation}\label{zhuji 2}
  \frac{q_{k_{n+1}(x)}^2(x)}{6q_{k_n(x)+3}^2(x)} \leq \frac{|J(\varepsilon_1(x) , \cdots, \varepsilon_n(x))|}{|J(\varepsilon_1(x) , \cdots, \varepsilon_{n+1}(x))|} \leq \beta^{l_{n+1}(x)+2}.
\end{equation}
On the other hand, Corollary \ref{cor} implies that
\[
q_{k_{n+1}(x)}(x) \geq 2 ^{\frac{k_{n+1}(x)-(k_n(x)+3)-1}{2}}q_{k_n(x)+3}(x).
\]
Hence, by (\ref{zhuji 2}), we obtain
\begin{equation*}
\frac{2^{k_{n+1}(x)-k_n(x)-4}}{6} \leq \frac{q_{k_{n+1}(x)}^2(x)}{6q_{k_n(x)+3}^2(x)} \leq \beta^{l_{n+1}(x)+2}.
\end{equation*}
That is,
\[
k_{n+1}(x)-k_n(x) \leq 4 + \frac{\log 6}{\log 2} +\frac{2\log \beta}{\log 2} + l_{n+1}(x)\frac{\log \beta}{\log 2}.
\]
In view of Lemma \ref{l-n chongduishu}, we deduce that for $\mathrm{P}$-almost all $x \in [0,1)$,
\begin{equation*}
\limsup\limits_{n \to \infty}\frac{k_{n+1}(x)-k_n(x)}{\sqrt{n\log\log n}} \leq 0.
\end{equation*}
Therefore, for $\mathrm{P}$-almost all $x \in [0,1)$,
\begin{equation*}\label{k-n jixian}
\lim\limits_{n \to \infty}\frac{k_{n+1}(x)-k_n(x)}{\sqrt{n\log\log n}} = 0.
\end{equation*}
\end{proof}

The following lemma is a key tool in the proof of Theorem \ref{law of the iterated logarithm}, which states that the limit in Theorem \ref{q law} can also be replaced by the limit over the subsequence $\{k_n(x)\}$.

\begin{lemma}\label{inequality 3}
Let $\beta >1$. Then for $\mathrm{P}$-almost all $x \in [0,1)$,
\begin{equation}\label{prove eq 1}
\limsup\limits_{n \to \infty} \frac{\log q_{k_n(x)}(x)- bk_n(x)}{\sigma_1\sqrt{2k_n(x)\log\log k_n(x)}} = 1
\end{equation}
and
\begin{equation}\label{prove eq 2}
\liminf\limits_{n \to \infty} \frac{\log q_{k_n(x)}(x)- bk_n(x)}{\sigma_1\sqrt{2k_n(x)\log\log k_n(x)}} = -1.
\end{equation}
\end{lemma}

\begin{proof}
Let $B_1$ and $B_2$ be the exceptional sets that the limit (\ref{theorem eq}) and Lemma \ref{inequality 2} do not hold respectively. Let $A = [0,1)\backslash (\mathbb{Q} \cup B_1 \cup B_2$), where $\mathbb{Q}$ denotes the set of rational numbers. Then $\mathrm{P}(A)=1$. Note that for any irrational $x \in [0,1)$,
\[
0 \leq k_1(x) \leq k_2(x) \leq \cdots\ \ \text{and}\ \  \lim\limits_{n \to \infty}k_n(x) = \infty.
\]
So for any $x \in A$ and any $i \geq 1$, there exists $n \geq 1$ such that $k_n(x) \leq i \leq k_{n+1}(x)$.

Therefore,
\begin{align}\label{zuo1}
&\frac{\log q_i(x)- bi}{\sigma_1\sqrt{2i\log\log i}} \leq
\frac{\log q_{k_{n+1}(x)}(x)- bk_n(x)}{\sigma_1\sqrt{2k_n(x)\log\log k_n(x)}}.
\end{align}

In view of (\ref{theorem eq}) and Lemma \ref{inequality 2}, we deduce that
\begin{equation}\label{lemma inequality 2}
\lim_{n \to \infty} \frac{b(k_{n+1}(x)-k_n(x))}{\sigma_1\sqrt{2k_{n+1}(x)\log\log k_{n+1}(x)}} =0.
\end{equation}
Since $k_n(x)/n \to a$ as $n \to \infty$, we have that

\begin{equation}\label{lemma inequality 12}
\lim_{n \to \infty} \frac{\sqrt{k_{n+1}(x)\log\log k_{n+1}(x)}}{\sqrt{k_n(x)\log\log k_n(x)}}=1
\end{equation}

Note that
\begin{align*}
\frac{\log q_{k_{n+1}(x)}(x)- bk_n(x)}{\sigma_1\sqrt{2k_{n}(x)\log\log k_{n}(x)}}
= \frac{\log q_{k_{n+1}(x)}(x)- bk_{n+1}(x)}{\sigma_1\sqrt{2k_{n}(x)\log\log k_{n}(x)}} + \frac{b(k_{n+1}(x)-k_n(x))}{\sigma_1\sqrt{2k_{n}(x)\log\log k_{n}(x)}}
\end{align*}
and
\begin{align*}
\frac{\log q_{k_{n+1}(x)}(x)- bk_{n+1}(x)}{\sigma_1\sqrt{2k_n(x)\log\log k_n(x)}}
= \frac{\log q_{k_{n+1}(x)}(x)- bk_{n+1}(x)}{\sigma_1\sqrt{2k_{n+1}(x)\log\log k_{n+1}(x)}} \cdot \frac{\sqrt{k_{n+1}(x)\log\log k_{n+1}(x)}}{\sqrt{k_n(x)\log\log k_n(x)}},
\end{align*}
combining this with (\ref{zuo1}), (\ref{lemma inequality 2}) and (\ref{lemma inequality 12}), we obtain that
\[
\limsup\limits_{n \to \infty} \frac{\log q_{k_n(x)}(x)- bk_n(x)}{\sigma_1\sqrt{2k_n(x)\log\log k_n(x)}} \geq
\limsup\limits_{i \to \infty}\frac{\log q_i(x)- bi}{\sigma_1\sqrt{2i\log\log i}} = 1,
\]
where the last equality is from Theorem \ref{q law}. Since the sequence $\{k_n(x)\}$ is a subsequence of $\{n\}$, we actually show that
\[
\limsup\limits_{n \to \infty} \frac{\log q_{k_n(x)}(x)- bk_n(x)}{\sigma_1\sqrt{2k_n(x)\log\log k_n(x)}} = 1.
\]

Similarly, we have that
\[
\liminf\limits_{n \to \infty} \frac{\log q_{k_n(x)}(x)- bk_n(x)}{\sigma_1\sqrt{2k_n(x)\log\log k_n(x)}} = -1.
\]
\end{proof}

Now we are ready to prove Theorem \ref{law of the iterated logarithm}.

\begin{proof}[Proof of Theorem \ref{law of the iterated logarithm}]

For any irrational $x \in [0,1)$ and $n \geq 1$, let
\begin{align*}
X_n(x) = - \frac{\log q_{k_n(x)}(x)-bk_n(x)}{\sigma_1\sqrt{2k_n(x)\log\log k_n(x)}},
\end{align*}
\begin{align*}
Y_n(x)= \frac{\sigma_1\sqrt{2k_n(x)\log\log k_n(x)}}{b\sigma\sqrt{2n\log\log n}}\ \ \ \ \ \ \text{and}\ \ \ \ \ \  Z_n(x)= \frac{\log q_{k_n(x)}(x)-abn}{b\sigma\sqrt{2n\log\log n}}.
\end{align*}

Therefore,
\begin{equation}\label{chongduishu}
\frac{k_n(x)-an}{\sigma\sqrt{2n\log\log n}} = X_n(x)\cdot Y_n(x) + Z_n(x),
\end{equation}
where $\sigma_1$ and $\sigma$ are related by the equation (\ref{sigma constant}).

Let $B_1$, $B_2$ and $B_3$ be the exceptional sets that the limit (\ref{theorem eq}), Lemma \ref{l-n chongduishu} and Lemma \ref{inequality 3} do not hold respectively. Let $A = [0,1)\backslash (B_1 \cup B_2 \cup B_3)$, then $\mathrm{P}(A)=1$.
For any $x \in A$, the equalities (\ref{theorem eq}) and (\ref{sigma constant}) show that $\lim\limits_{n \to \infty} Y_n(x) = 1$. The second equality of Lemma \ref{l-n chongduishu} implies that $\lim\limits_{n \to \infty} Z_n(x) = 0$. By Lemma \ref{inequality 3}, we have that
\[
\limsup_{n \to \infty} X_n(x) = 1\ \ \ \ \  \text{and}\ \ \ \ \  \liminf_{n \to \infty} X_n(x) = -1.
\]
Combining this with (\ref{chongduishu}), we obtain that
\[
\limsup_{n \to \infty} \frac{k_n(x)-an}{\sigma\sqrt{2n\log\log n}} = 1\ \ \ \ \  \text{and}\ \ \ \ \  \liminf_{n \to \infty} \frac{k_n(x)-an}{\sigma\sqrt{2n\log\log n}} = -1.
\]

\end{proof}

{\bf Acknowledgement}
The work was supported by NSFC 11371148, 11201155 and Guangdong Natural Science Foundation 2014A030313230.

\end{document}